\documentclass[11pt,reqno]{amsart}
\usepackage{amssymb}
\usepackage[T1]{fontenc}
\usepackage[dvipsnames]{xcolor}
\usepackage{palatino}
\input amssym.def
\usepackage{amsmath, amsfonts}
\usepackage{enumitem}
\usepackage{amscd, mathtools}
\usepackage[mathscr]{eucal}
\usepackage{palatino}
\usepackage{hyperref}
\hypersetup{colorlinks,citecolor= black,linkcolor= black,urlcolor=blue}

\newfont{\cyrr}{wncyr10}

\newcommand{\C}{{\mathbb C}}

\newcommand{\pL}{{L^2 \leq p \leq \exp(\log^2 L)}}

\newcommand{\thmref}[1]{Theorem~\ref{#1}}
\newtheorem{thm}{Theorem}

\newtheorem{lem}[thm]{Lemma}
\newtheorem{cor}[thm]{Corollary}
\newtheorem{prop}[thm]{Proposition}

\newtheorem{rmk}{Remark}[section]

\newcommand{\corref}[1]{Corollary~\ref{#1}}
\newcommand{\propref}[1]{Proposition~\ref{#1}}

\newcommand{\lemref}[1]{Lemma~\ref{#1}}

\newcommand{\rmkref}[1]{Remark~\ref{#1}}

\begin{document}
\title[Extreme Values]{Extreme values of $L$-functions of newforms}

\author{Sanoli Gun and Rashi Lunia}

\address{Sanoli Gun and Rashi Lunia \\ \newline
	The Institute of Mathematical Sciences, A CI of Homi Bhabha National Institute, 
	CIT Campus, Taramani, Chennai 600 113, India.}

\email{sanoli@imsc.res.in}
\email{rashisl@imsc.res.in}



\begin{abstract}
In 2008, Soundararajan showed that there exists a normalized Hecke eigenform
$f$  of weight $k$ and level one such that 
$$
L(1/2, f ) 
~\geq~
\exp\Bigg( (1 + o(1)) \sqrt{\frac{2\log k}{\log\log k}  }\Bigg)
$$ 
for sufficiently large $k \equiv 0 \pmod{4}$. In this note, we show that for 
any $\epsilon>0$ and 
for all sufficiently large $k \equiv 0 \pmod{4}$, the number of 
normalized Hecke eigenforms of
weight $k$ and level one for which
$$
L(1/2,  f ) 
~\geq~
\exp\left(1.41\sqrt{ \frac{ \log k }{\log\log k} }\right)
$$
is $\gg_{\epsilon} k^{1-\epsilon}$. 
For an odd fundamental discriminant $D$, let
$B_{k}(|D|)$ be the set of all cuspidal normalized Hecke
eigenforms of weight $k$ and level dividing $|D|$.
When the real primitive Dirichlet character $\chi_D$ satisfies
$\chi_D(-1)= i^k$, we investigate the number 
of $f \in B_{k}(|D|)$ for which $L(1/2, f \otimes \chi_D)$ 
takes extremal values. 
\end{abstract}

\maketitle 

\section{Introduction and Statement of results}

\smallskip

Throughout the article, $\epsilon$ denotes an 
arbitrarily small positive quantity, which may change time to time.
For an odd square free integer $q \ge 1$, let 
$S_k(q)$ be the space of cusp forms of even weight $k$ on
$\Gamma_0(q)$.
Also let $H_k(q)$ be the set of cuspidal normalized Hecke
eigenforms of weight $k$ and level $q$ and
$B_k(q) = \cup_{q' | q} H_{k}(q')$. 
For any $f \in H_k(q)$, the Fourier expansion of $f$ is given by 
$$
f(z)=\sum_{n=1}^{\infty}\lambda_f(n)n^{(k-1)/2}e(nz),
$$
where $\lambda_f(1)=1$ and $e(z)=e^{2\pi i z}$ for any $z \in \C$ with $\Im(z)> 0$.
The $L$-function associated to $f$ is given by
$$
L(s, f)=\sum_{n=1}^\infty \frac{\lambda_f(n)}{n^s}, \qquad \Re(s) > 1.
$$
This function has an analytic continuation to the entire complex plane $\C$
and satisfies the functional equation 
$$
\Lambda(s, f)= i^k \eta_f \Lambda(1-s, f),
$$
where $\eta_f$ is the eigenvalue of the Fricke involution $W_q$ and
$$
\Lambda(s, f)
=
\Big(\frac{\sqrt{q}}{2\pi}\Big)^{s}~\Gamma\Big(s + \frac{k-1}{2}\Big)L(s, f)
$$
(see page 82 of \cite{ILS} for details).
The value of $L(s, f)$ at the central critical point $1/2$ has been 
studied extensively
in different aspects. In \cite{KS}, Soundararajan showed 
that for any sufficiently large weight 
$k \equiv 0 \pmod{4}$, there exists an $f \in H_k(1)$ for which
$$
L\big(1/2,  f \big) 
~\geq~
\exp\Bigg( (1 + o(1)) \sqrt{\frac{2\log k}{\log\log k}  }\Bigg).
$$ 
In this article, we want to investigate the number of
$f \in H_k(1)$ which takes such extremal values.
More precisely, we prove the following theorem.
\begin{thm}\label{thm1}
Let  
$$
\mathcal{S}(k)
=
\Bigg\{ f \in H_k(1)  ~\Big\vert~ 
L\big(1/2,  f \big) ~\geq~
\exp\Bigg(1.41\sqrt{ \frac{ \log k }{\log\log k } }\Bigg)\Bigg\}.
$$
For all sufficiently large $k \equiv 0 \pmod{4}$, we have
$$
|\mathcal{S}(k)| 
~\gg~ \frac{k}{ (\log k)^{11}  \exp\Big(\frac{3\log k}{2\log\log k} \Big)} ~,
$$
where the implied constant is absolute.
\end{thm}

\begin{rmk} \hfill 
\begin{itemize} 
\item 
We note that the constant $1.41$ can be replaced by any constant strictly less 
than $\sqrt{2}$. 

\item
Kohnen, as a corollary to his result in \cite{WK}, deduced that the number of  
$f \in H_k(1)$ for which  $L(1/2, f) \ne 0$ is $\gg k^{1/2}$.
Iwaniec and Sarnak \cite{IS} studied the  proportion of $f \in H_k(1)$ for which 
$L(1/2, f) \ne 0$ with an extra
averaging over the weight. Later Lau and Tsang \cite{LT}
showed that the number of $f \in H_k(1)$ for which  $L(1/2, f) \ne 0$
is $\gg  k / \log^2 k$. In 2015, Luo \cite{WL} refined this result to show that 
this number is $\gg k$.

\item
One can compare the lower bound in \thmref{thm1} with the
upper bound 
$$
|S(k)| \ll \frac{k\log^3k}{\exp\left(2.82\sqrt{\frac{\log k}{\log\log k}}\right)}.
$$
This follows from the observation that
 $$
 \sum_{ f \in S(k)} L^2(1/2, f) 
 ~\ll~  
 \sum_{ f  \in H_k} L^2(1/2, f) 
 ~\ll~ k \log^3 k 
 \phantom{m} \text{ as } k \to \infty.
 $$
 Above we have used the fact that $\sum_{f \in H_k}\frac{L^2(1/2, f)}{\omega(f)^*}$ is asymptotic to
 $c k \log k$ for some constant $c$ as $k \to \infty$ with $k \equiv 0 \pmod{4}$ (see \cite{RS})
and the bound on $\omega(f)^*$ (see Preliminaries).
Moreover, under the Keating-Snaith 
conjectures \cite{KeS} (see also \cite{CFKRS}), for any positive integer
$r$, there exists a constant $c(r)>0$ such that  
\begin{equation}
\sum_{f \in H_k}\frac{L^r(1/2, f)}{\omega(f)^*} \sim c(r)k (\log k)^{r(r-1)/2}
\end{equation}
as $k \to \infty$ with $k \equiv 0 \pmod{4}$ (see \cite{RS}).
This implies that for any $\epsilon > 0$, we have 
$$
|S(k)| 
\ll \frac{k(\log k)^{r(r-1)/2+2}}
{\exp\left(1.41 r \sqrt{\frac{\log k}{\log\log k}}\right)} 
\ll_{\epsilon} \frac{k}{\exp\left((1.41 r - \epsilon)
\sqrt{\frac{\log k}{\log\log k}}\right)}.
$$
\end{itemize}
\end{rmk}
Next we investigate the values of twists of modular $L$-functions at $s =1/2$.
Let $D$ be an odd fundamental discriminant and $\chi_D$ the real 
primitive character associated with it. For any $f \in B_k(|D|)$, there exists
$D'|D$ such that $f \in H_k(|D'|)$. The cusp form 
$$
f \otimes \chi_D(z) 
~=~
\sum_{n=1}^{\infty}\lambda_f(n)\chi_D(n)n^{(k-1)/2}e(nz)
$$
is a newform in $S_k(|D|^2)$ (see Theorem 6 of \cite{AL}).
The $L$-function attached to the twisted form $f \otimes \chi_D$ is given by
$$
L(s, f\otimes\chi_D)=\sum_{n=1}^\infty \frac{\lambda_f(n)\chi_D(n)}{n^s}, 
\qquad \Re(s)>1.
$$
The associated completed $L$-function 
$$
\Lambda(s, f\otimes \chi_D)=\Big(\frac{|D|}{2\pi}\Big)^s
\Gamma\Big(s+\frac{k-1}{2}\Big)L(s, f\otimes \chi_D)
$$
is entire and satisfies the functional equation
$$
\Lambda(s, f\otimes \chi_D)=\chi_D(-1)i^k\Lambda(1-s,f\otimes \chi_D)
$$
(see \cite{CI, MR} for further details).
We shall assume that $\chi_D(-1)=i^k$ as otherwise the central critical values 
$L(1/2, f \otimes \chi_D)$ vanish. In this set-up, we have the following theorem.

\begin{thm}\label{thm2}
Let $k \ge 2$ be an even integer and $D$ an 
odd fundamental discriminant such that ${\chi_D(-1)=i^k}$. Also 
let
$$
\mathcal{S}(k,D)
=
\left\{f \in B_k(|D|) ~\Big\vert~ L\left(1/2, f\otimes\chi_D \right) 
~\geq~
\exp\left(1.41\sqrt{ \frac{ \log(k|D|) }{\log\log(k|D|) } }\right)\right \}.
$$
For any $\epsilon > 0$, we have
$$
|\mathcal{S}(k,D)| 
\gg_{\epsilon, D}
k^{1-\epsilon}.
$$
\end{thm}

\begin{rmk}\label{reqcor}
The proof of \thmref{thm2} also implies the following statement.
For an even integer $k \ge 2$, an 
odd fundamental discriminant $D$ satisfying ${\chi_D(-1)=i^k}$
and any $\epsilon > 0$, we have
$$
\# \left\{f \in H_k ~\Big\vert~ L\left(1/2, f\otimes\chi_D \right) 
~\geq~
\exp\left(1.41\sqrt{ \frac{ \log(k|D|) }{\log\log(k|D|) } }\right)\right \}
~\gg_{\epsilon, D}~ 
k^{1-\epsilon}.
$$
The essential point is to note that  \eqref{PET} holds for $H_k$ and
consequently \lemref{chiasym} holds when
$B_k(|D|)$ is replaced by~$H_k$.
\end{rmk}

Now let $S_{(k+1)/2}(4)$ be the space of cusp forms of half-integer weight $(k+1)/2$ 
for $\Gamma_0(4)$ and $S_{(k+1)/2}^+(4)$ its "plus subspace" (see \cite{WK} for details). 
For  $g \in S_{(k+1)/2}^+(4)$, let 
$$
g(z) = \sum_{n=1}^{\infty} c_g(n) e(n z)
$$
be the Fourier expansion of $g$.  By explicit Waldspurger's formula \cite{KZ, JW},
Fourier coefficients of $g$ are related to the values 
of $L(1/2,  f \otimes \chi_D)$, where $f$ is the image of $g$ under the Shimura
correspondence \cite{GS} and $D$ is a fundamental discriminant with $(-1)^k D >0$. 
 In recent works \cite{GKS, JLS}, it was shown that the 
Fourier coefficients $c_g(|D|)$ occasionally become large in terms of $|D|$.
As an applications of our results, we get the following corollary.

\begin{cor}\label{cor1} 
Let $k$ be an even integer and $D$ an odd fundamental 
discriminant with ${\chi_D(-1)=i^k}$.
For any $\epsilon>0$, the number of Hecke
eigenforms $g \in S_{(k+1)/2}^+(4)$ normalized 
in the sense
$$
\|g \|^2
~=~
\int_{\Gamma_0(4)\backslash \mathcal{H}} |g(z)|^2y^{(k+1)/2}\frac{dxdy}{y^2}
~=~1,
$$
for which 
$$
c_g(|D|)^2
~\geq~
\frac{\Gamma(k/2)}{\pi^{k/2}}|D|^{(k-1)/2} 
\exp\left(1.41\sqrt{ \frac{ \log(k|D|) }{\log\log(k|D|) } }\right)
$$
is $ \gg_{\epsilon, D} k^{1 - \epsilon}$. 
\end{cor}

\begin{rmk}
For a Hecke eigenform $g \in S_{(k+1)/2}^+(4)$, it  is not necessarily true that 
its first Fourier coefficient $c_g(1)$ is non-zero. The above corollary 
proves that for any $\epsilon >0$, there are $\gg_{\epsilon} k^{1 - \epsilon}$ 
normalized Hecke eigenforms 
in $S_{(k+1)/2}^+(4)$ for which not only is $|c_g(1)|$ non-zero, it is rather large.
\end{rmk}

The article is organized as follows; in section 2 we recall some preliminaries,
in section 3 we prove some introductory results 
and finally in section 4 we complete the proofs of the main theorems.

\smallskip

\section{Preliminaries}

\medskip

\subsection{Petersson trace formula and modular $L$-function}
Let $q \ge 1$ be an odd square free integer. Let $S_k(q), B_k(q)$ and $H_k(q)$  
be as in the introduction. For brevity, from now onwards, we denote 
$S_k(1)$ by $S_k$ and $H_k(1)= B_k(1)$ by $H_k$.
For $f \in B_k(q)$, set
\begin{equation*}
\omega(f)^* ~=~q~ \sum_{n \atop (n, q)=1} \frac{\lambda_f(n^2)}{n}.
\end{equation*}
It is known that (see \cite{CI}, \cite{HL} and \cite{ILS})
\begin{equation}\label{omega}
k^{-\epsilon} ~\ll_{\epsilon, q}~  \omega(f)^* ~\ll_{\epsilon, q}~  k^{\epsilon}
\end{equation}
as $k \to \infty$. For $q=1$, we know that (see equation 16(c) of \cite{KS})
\begin{equation}\label{omega2}
(\log k)^{-2} ~\ll~  \omega(f)^* ~\ll~ (\log k)^2
\end{equation}
as $k \to \infty$. For integers $m , n \geq 1$ with $(mn, q)=1$, by the Petersson trace formula 
(see Lemma 2.7 of \cite{ILS}), we have 
$$
\frac{12}{k-1}\sum_{f \in B_k(q)}\frac{\lambda_f(m)\lambda_f(n)}{\omega(f)^*}
~=~
\delta_{m, n} ~+~  2\pi i^k\sum_{c=1}^\infty\frac{S(m,n; cq)}{cq}
J_{k-1}\left(\frac{4\pi\sqrt{mn}}{cq}\right),
$$
where $\delta_{m,n}$ is the Dirac delta function, $J_{k-1}$ is the Bessel function 
and $S(m, n ; .)$ is the Kloosterman sum.
From equation 18 of \cite{KS}, for $\sqrt{mn} \leq kq/(40\pi)$ with $(mn, q) =1$, 
we have 
\begin{equation}\label{PET}
\frac{12}{k-1}\sum_{f \in B_k(q)}\frac{\lambda_f(m)\lambda_f(n)}{\omega(f)^*}
~=~
\delta_{m,n} ~+~ O(e^{-k}).
\end{equation}
We will use Petersson trace formula along with third moments 
of the $L$-functions attached to cuspidal Hecke eigenforms at the critical point
to deduce our results.
When $q=1$, Peng \cite{ZP} proved the following upper bound.

\begin{thm}\label{peng}
Let $k$ be an even integer and $H_k$ be as defined in the introduction. 
For any $\epsilon >0$, we have 
$$
\sum_{f \in H_k}
L(1/2, f)^{3} ~\ll_{\epsilon}~ k^{1+\epsilon}.
$$
\end{thm}
Young \cite{MY} extended this result for odd square free integers $q$ and proved 
the following bound. 

\begin{thm}\label{young}
Let $k$ be an even integer and $q$ an odd square free integer with $\chi_q(-1)= i^k$.
For any $\epsilon>0$, we have
$$
\sum_{f \in B_k(q)}
L(1/2, f\otimes\chi_q )^{3} 
~\ll_{\epsilon, q}~ 
k^{1+\epsilon}.
$$
\end{thm}

In 2020, Frolenkov \cite{DF} refined the result of Peng to prove the following theorem.

\begin{thm}\label{fro}
Let $k$ be an even integer and $H_k$ be as defined in the introduction. We have
$$
\sum_{f \in H_k} \frac{L(1/2, f)^{3}}{\omega(f)^{*} }  ~\ll~ k \log^{9/2}k.
$$
\end{thm}

\smallskip

\subsection{Half-integer weight modular forms and explicit Waldspurger formula}  

Let $S_{(k+1)/2}(4)$ be 
the space of cusp forms of half-integer weight 
${(k+1)/2}$ for $\Gamma_0(4)$.  
The "plus space" $S_{(k+1)/2}^+(4)$ is the subspace of $S_{(k+1)/2}(4)$ 
consisting of forms whose $n$-th Fourier coefficients 
vanish unless $(-1)^{k/2} n\equiv {0,1}\pmod 4$. 
The theory of such forms was developed in \cite{WK}, \cite{WK1}, \cite{SN} and \cite{GS}.
The Shimura correspondence associates to every Hecke eigenform 
$g \in S_{(k + 1)/2}^+(4)$ a Hecke eigenform $f \in S_k$ (see \cite{WK, WK1} for details).  
Let 
$$
g(z)~=~\sum_{n=1}^{\infty} c_g(n) e(nz)
$$
be the Fourier expansion of $g$ which is normalized in the sense
\begin{equation}\label{norm}
\|g\|^2
~=~
\int_{\Gamma_0(4)\backslash \mathcal{H}}|g(z)|^2~y^{(k+1)/2}~\frac{dxdy}{y^2}
~=~1.
\end{equation}
By the explicit Waldspurger's formula (see equation 1.4 of \cite{HI}, also see \cite{KZ, JW}), we have 
\begin{equation}\label{2.51}
c_g(|D|)^2
~=~
\frac{\Gamma(k/2)}{\pi^{k/2}}~| D|^{(k-1)/2}~L\Big(1/2, f \otimes \chi_D \Big).
\end{equation}

\begin{rmk}
It follows from \eqref{2.51} that $L(1/2, f\otimes\chi_{D} )$ is non-negative.
\end{rmk}

\smallskip

\section{Introductory Results}

\smallskip
From now on, $p$ will denote a rational prime, $D$ an odd fundamental discriminant 
and $\chi_D$ the real primitive character associated with $D$. 
With these notations, for $N \ge 1$ (to be chosen later), we define a multiplicative function 
$r_D$ such that $r_D(p^n)=0$ for $n>1$ and
\begin{equation}
r_D(p)=
\begin{cases}
\frac{\chi_D(p)}{\sqrt{p}\log{p}}L & \text{when~~
$L^2 \leq p \leq \exp(\log^2L)$}\\
0 & \textrm{otherwise},
\end{cases}
\end{equation}
where $L=\sqrt{\log{N}\log\log{N}}$.
For $f\in B_k(|D|)$, the resonator is defined as
\begin{equation}\label{resonator}
R(f\otimes \chi_D)=\sum_{m \leq N} r_D(m)\lambda_f(m).
\end{equation}
When $D=1$, we note that $f\otimes \chi_D =f$ and in this case, we shall 
denote $r_D(\cdot)$ by $r(\cdot)$. For $c > \frac12$, consider the integral
$$
I 
~=~ 
\frac{1}{2 \pi i} \int\displaylimits_{c - i\infty}^{c + i\infty}  \left(\frac{|D|}{2\pi} \right)^{s} 
\frac{\Gamma(s + \frac{k}{2})}{\Gamma(\frac{k}{2})}
L\left(s + \frac12, ~f\otimes \chi_D \right) \frac{ds}{s}.
$$
Now if we move the line of integration to $-c$, we get
$$
I
~=~ 
L\left(\frac12, ~f \otimes \chi_D \right) ~+~
\frac{1}{2 \pi i} 
\int\displaylimits_{-c - i \infty}^{-c + i \infty}  \left(\frac{|D|}{2\pi} \right)^{s} 
\frac{\Gamma(s + \frac{k}{2})}{\Gamma(\frac{k}{2})}
L\left(s + \frac12, ~f\otimes \chi_D \right) \frac{ds}{s}.
$$
By changing the variable $s$ to $-s$ and then applying functional equation
for $L(\frac12 - s, ~f \otimes \chi_D)$ along with $\chi_D(-1)= i^k$, we get
\begin{equation}\label{Approxfn}
L(1/2, f \otimes \chi_D)
~=~
2\sum_{n=1}^{\infty}
\frac{\lambda_f(n)\chi_D(n)}{\sqrt{n}}
~V(n/|D|),
\end{equation}
where 
$$
V(x)
~=~
\frac{1}{2\pi i}\int_{(c)}(2\pi)^{-s}~\frac{\Gamma\left(s+\frac{k}{2}\right)}
{\Gamma\left(\frac{k}{2}\right)}~x^{-s}~\frac{ds}{s}.
$$
Moving the line of integration to $k/2$ and respectively to $1-k/2$ and applying
the asymptotic relation (see page 93 of \cite{MM})
$$
|\Gamma( \sigma + it)| ~\sim~ e^{-\frac{\pi}{2} |t| } |t|^{\sigma -\frac{1}{2}} \sqrt{2\pi}
$$
when $\sigma$ is fixed and $|t| \to \infty$, we get
\begin{eqnarray}\label{bd1}
V(x) ~\ll~ \Big( \frac{k}{2\pi x} \Big)^{k/2} 
&\phantom{m}{\rm and} \phantom{m} & 
V(x)
~=~
1 ~+~ O\left(\frac{(2\pi x)^{k/2-1}}{\Gamma(\frac{k}{2})}\right)
\end{eqnarray}
respectively. See also equation 20(e) of \cite{KS}.

\begin{lem}\label{chiasym}
Let $k$ be an even integer and $D$ an odd fundamental discriminant
with $\chi_D(-1)=i^k$. We have
\begin{equation}\label{small}
\frac{12}{k-1}\sum_{f \in B_k(|D|)}\frac{R(f \otimes \chi_D)^2}{\omega(f)^*}
~\sim~ 
\prod_p \Big(1 + r_D(p)^2 \Big)
\end{equation}
as $k \to \infty$. We also have
\begin{equation}\label{big}
\frac{12}{k-1}\sum_{f \in B_k(|D|)}\frac{R(f \otimes \chi_D)^2}{\omega(f)^*}
~L\Big(\frac12, ~f \otimes \chi_D \Big)
~\sim~
 2\prod_p \Big(1+r_D(p)^2\Big(1+\frac{1}{p}\Big)+2\frac{r_D(p)\chi_D(p)}{\sqrt{p}}\Big)
\end{equation}
as $k \to \infty$.
Further, the ratio of \eqref{big} to \eqref{small} is 
$\exp\left((1+o(1))\sqrt{\frac{2\log(k|D|)}{\log\log(k|D|)}}\right)$ as $k \to \infty$.
\end{lem}

\begin{proof}
The proof for the case $D=1$ is given in \cite{KS}.  We give a sketch of the
proof for an arbitrary odd fundamental discriminant $D$.
Note that
$$
\mathcal{I} :=
\frac{12}{k-1}\sum_{f \in B_k(|D|)}\frac{R(f \otimes \chi_D)^2}{\omega(f)^*}
~=~ 
\frac{12}{k-1}\sum\displaylimits_{m,n \leq N}r_D(m)r_D(n)
\sum_{f \in B_k(|D|)}\frac{\lambda_f(m)\lambda_f(n)}{\omega(f)^*}.
$$
When $N\leq \frac{k|D|}{40\pi}$, applying \eqref{PET} and the Cauchy-Schwarz 
inequality, we get
\begin{eqnarray}
 \mathcal{I} 
 = \sum\displaylimits_{m,n \leq N}r_D(m)r_D(n)
\left(\delta_{m,n}+O(e^{-k})\right) 
&=& 
\sum\displaylimits_{n \leq N}r_D(n)^2
~+~
O(e^{-k}N\sum\displaylimits_{n \leq N}r_D(n)^2) \nonumber \\
&=& 
\sum\displaylimits_{n \leq N}r_D(n)^2
(1 ~+~ O_D(ke^{-k})) \label{first}.
\end{eqnarray}
Now, using Rankin's trick for any $\alpha > 0$, we see that 
\begin{eqnarray}\label{ratio1}
\sum\displaylimits_{n \leq N}r_D(n)^2
&=&
\sum\displaylimits_{n=1}^{\infty}r_D(n)^2 
~+~
O\Big(N^{-\alpha}\sum\displaylimits_{n > N}r_D(n)^2n^{\alpha}\Big)  \nonumber\\
&=&
\prod\displaylimits_{p}(1+r_D(p)^2) 
~+~
O\Big(N^{-\alpha}\prod\displaylimits_{p}(1+r_D(p)^2p^{\alpha})\Big).
\end{eqnarray}
Note that
$$
\prod\displaylimits_{p}(1+\frac{r_D(p)^2(p^{\alpha}-1)}{1+r_D(p)^2})
~=~
 \exp\Big( \sum\displaylimits_{p}\log (1+\frac{r_D(p)^2(p^{\alpha}-1)}{1+r_D(p)^2}) \Big).
$$
Choosing $\alpha=(\log L)^{-3}$, we see that $p^\alpha <2$ for sufficiently large $N$
and using $\log (1+x) \leq x$ for $0 \le x < 1$, we have 
$$
\sum\displaylimits_{\pL}
\frac{r_D(p)^2(p^{\alpha}-1)}{1+r_D(p)^2}
~\leq~
\sum\displaylimits_{\pL} r_D(p)^2(p^{\alpha}-1).
$$
Again by noting that $e^x- 1 \leq  x + x^2$ when $x<1$ and using the bound on $p$, 
we get
$$
p^{\alpha} - 1
~\leq ~
\alpha \log p+ (\alpha \log p)^2
~\leq ~
\alpha \log p \Big( 1 + \frac{1}{\log L} \Big).
$$
Putting $L=\sqrt{\log N \log\log N}$ and using the bounds 
$\frac{1}{2}\log\log N \leq \log L \leq \log\log N$,
we get 
\begin{eqnarray*}
\sum\displaylimits_{\pL} r_D(p)^2(p^{\alpha}-1)
&\le&
\frac{ \alpha L^2}{2\log L}
\Big( 1 - \frac{1}{2\log L}\Big)
~\le~ 
\alpha \log N
\Big(1 - \frac{1}{2\log \log N}\Big)
\end{eqnarray*}
for sufficiently large $N$. Therefore the ratio of the error term to the main term
in \eqref{ratio1}  is
$$
O\Big( \exp\Big(-\frac{\alpha\log N}{\log \log N}\Big)\Big)
$$
for sufficiently large $N$. Thus putting everything together in \eqref{first}, we get  
$$
\frac{12}{k-1}\sum_{f \in B_k(|D|)}\frac{R(f \otimes \chi_D)^2}{\omega(f)^*}
~=~
(1+o(1))\prod\displaylimits_{\pL}(1+r_D(p)^2).
$$
This completes the proof of \eqref{small}. 

\smallskip

We now give a sketch of the proof of \eqref{big}.  Using \eqref{Approxfn}, we know that
$$
L\big(\frac12, ~f \otimes \chi_D \big)
~=~
2\sum_{n \leq 2k |D|}
\frac{\lambda_f(n)\chi_D(n)}{\sqrt{n}}V\Big(\frac{n}{|D|}\Big) 
~+~ 
2 \sum_{n>2k|D|} \frac{\lambda_f(n)\chi_D(n)}{\sqrt{n}}V\Big(\frac{n}{|D|}\Big).
$$
Since
$V(x) \ll \left(\frac{k}{2\pi x}\right)^{k/2}$, we can write
$$
\sum_{n>2k|D|}
\frac{\lambda_f(n)\chi_D(n)}{\sqrt{n}}
V (\frac{n}{|D|})
\ll 
\sum_{n>2k|D|}
\frac{\lambda_f(n)\chi_D(n)}{\sqrt{n}}(\frac{k|D|}{2\pi n})^{k/2}
\ll_D
k\sum_{n>2k|D|}
\frac{\lambda_f(n)}{n^{3/2}} (\frac{1}{4\pi})^{k/2 -1}
\ll_D e^{-k}.
$$
Hence
$$
L\Big(\frac12, ~f \otimes \chi_D \Big)
~=~
2\sum_{n \leq 2k |D|}
\frac{\lambda_f(n)\chi_D(n)}{\sqrt{n}}V\Big(\frac{n}{|D|} \Big) 
~+~ 
O_D(e^{-k}).
$$
Thus
\begin{eqnarray}\label{c1}
&&
\frac{6}{k-1}\sum_{f \in B_k(|D|)}\frac{R(f \otimes \chi_D)^2 
L(\frac12, f \otimes \chi_D)}{\omega(f)^*} \nonumber \\
 &=&
  \frac{12}{k-1}\sum_{n=1}^{2k|D|} \frac{\chi_D(n)}{\sqrt{n}} V(\frac{n}{|D|})~
\sum_{f \in B_k(|D|)}\frac{R(f \otimes \chi_D)^2\lambda_f(n)}{\omega(f)^*}
~+~
O_D(~e^{-k} \prod\displaylimits_{p}(1 + r(p)^2  ) ~) \nonumber\\
&=&
I_1 ~+~ O_D(~e^{-k} \prod\displaylimits_{p}(1 + r(p)^2  ) ~), ~ \text{ say} .
\end{eqnarray}
Using Hecke relations and Petersson trace formula when $N\leq \frac{\sqrt{k|D|}}{200}$
and $(n, D)=1$, we have
\begin{align*}
\sum_{f \in B_k(|D|)} \frac{R(f \otimes \chi_D)^2\lambda_f(n)}{\omega(f)^*}
=&
\sum_{f \in B_k(|D|)}
\sum\displaylimits_{m_1 \le N, \atop m_2 \leq N} 
\frac{\lambda_f(n) \lambda_f(m_1)
\lambda_f(m_2)r_D(m_1)r_D(m_2) }{\omega(f)^*} \\
=&
\sum_{f \in B_k(|D|)}
\sum\displaylimits_{m_1 \le N, \atop m_2 \leq N}
\frac{\lambda_f(n)  r_D(m_1)r_D(m_2) }{\omega(f)^*}
\sum\displaylimits_{d|(m_1,m_2) \atop (d, D)=1}
\lambda_f(\frac{m_1m_2}{d^2}) \\
=&~~
\frac{k-1}{12} \sum\displaylimits_{m_1 \le N, \atop m_2 \leq N} r_D(m_1)r_D(m_2)
\sum\displaylimits_{d|(m_1,m_2)}
(\delta_{\frac{m_1m_2}{d^2},n} ~+~ O(e^{-k})~).
\end{align*}
Hence
\begin{equation}\label{e1}
I_1
=
\sum\displaylimits_{m_1,m_2 \leq N} r_D(m_1)r_D(m_2)
\sum\displaylimits_{d|(m_1,m_2)}
\Big(\chi_D(m_1m_2)\frac{d}{\sqrt{m_1m_2}}
V(\frac{m_1m_2}{d^2|D|} )
~+~ 
O_D(k e^{-k/8}) \Big).
\end{equation}
Here we have used the second bound in \eqref{bd1}
to bound $V(n/|D|)$ when $n \le k|D|/4\pi e$ and used
the first bound in \eqref{bd1} when $k|D|/4\pi e <  n \leq 2k|D|$.
Thus the error term in $I_1$ is bounded by
\begin{equation}\label{e2}
\ll_D~~  k^3 e^{-\frac{k}{8}}
\prod\displaylimits_{p}(1 + r(p)^2).
\end{equation}
Using the second bound in \eqref{bd1}, we see that the main term of $I_1$ 
is equal to
$$
\sum\displaylimits_{m_1,m_2 \leq N}
\frac{ r_D(m_1)r_D(m_2) \chi_D(m_1m_2)}{ \sqrt{m_1m_2}}
\Big( \sigma(~(m_1,m_2)~)
~+~
O\Big( \Big(\frac{2\pi N^2}{|D|}\Big)^{\frac{k}{2}-1} 
\frac{N^2}{\Gamma(\frac{k}{2}) } \Big) \Big),
$$
where $\sigma(n) = \sum_{d | n} d$. When $N \leq \sqrt{k|D|}/200$, using Stirling's formula,
we see that the above sum is equal to
\begin{eqnarray}\label{e3}
&&
\sum\displaylimits_{m_1,m_2 \leq N} \frac{r_D(m_1)r_D(m_2)
\chi_D(m_1m_2)}{\sqrt{m_1m_2}} \Big(\sigma((m_1,m_2))
~+~
O_D\Big( \sqrt{k} \Big(\frac{2\pi k}{100}\Big)^{\frac{k}{2}}
\Big(\frac{2e}{k}\Big)^{\frac{k}{2}}\Big) \Big) \nonumber \\
&=&
\sum\displaylimits_{m_1,m_2 \leq N} \frac{r_D(m_1)r_D(m_2)
\chi_D(m_1m_2)\sigma((m_1,m_2))}{\sqrt{m_1m_2}}
~+~
O_D\Big(k e^{-\frac{k}{2}}\sum\displaylimits_{m \leq N} r(m)^2\Big).
\end{eqnarray}
Putting the terms \eqref{e2} and \eqref{e3} in \eqref{e1},  we get
\begin{eqnarray}
I_1
&=&
\sum_{d \leq N}\sum\displaylimits_{\substack{t, s\leq \frac{N}{d}\\{(ts, ~d)=1}}}
r_D(t)r_D(s)r_D(d)^2\frac{\chi_D(ts)}{\sqrt{ts}}
~+~ 
O_D(k^3 e^{-\frac{k}{8}}\sum\displaylimits_{m \leq N} r(m)^2 )\nonumber\\
&=&
\sum_{d \leq N}r_D(d)^2
(\sum\displaylimits_{\substack{t\leq \frac{N}{d}\\{(t, ~d)=1}}}
\frac{r_D(t)\chi_D(t)}{\sqrt{t}} ~)^2
~+~
 O_D(k^3e^{-\frac{k}{8}}\sum\displaylimits_{m \leq N} r(m)^2 ).\label{final}
\end{eqnarray}
Now using Rankin's trick, for $\alpha> 0$, the main term in \eqref{final} becomes
\begin{eqnarray}\label{c2}
&&
\sum_{d \leq N} r_D(d)^2 
\Big(
\sum\displaylimits_{\substack{t\geq 1\\{(t,d)=1}}}
\frac{r_D(t)\chi_D(t)}{\sqrt{t}}
~-~
\sum\displaylimits_{\substack{t>\frac{N}{d}\\{(t,d)=1}}} 
\frac{r_D(t)\chi_D(t)}{\sqrt{t}} 
~\Big)^2 \nonumber\\
&=&
\sum_{d \leq N}r_D(d)^2
\Big\{\prod\displaylimits_{p \nmid d}
\Big(1 + \frac{r_D(p)\chi_D(p)}{\sqrt{p}} \Big)^2
~+~  
O\Big( \big(\frac{d}{N}\big)^{\alpha}
\Big(\sum\displaylimits_{\substack{t >\frac{N}{d}\\{(t,d)=1}}}
\frac{r_D(t)\chi_D(t)}{\sqrt{t}}t^{\alpha/2}\Big)^2
\Big) 
\Big. \nonumber\\
&& \phantom{mmmmmm}+ \phantom{m}
\Big. 
O\Big(\big(\frac{d}{N}\big)^{\alpha}
\prod\displaylimits_{p \nmid d}
\Big(1+ \frac{r_D(p)\chi_D(p)}{\sqrt{p}}\Big)
\sum\displaylimits_{\substack{t>\frac{N}{d}\\{(t,d)=1}}}
\frac{r_D(t)\chi_D(t)}{\sqrt{t}}t^{\alpha}
\Big)
\Big\} \nonumber\\
&=&
\sum_{d \leq N} r_D(d)^2
\Big\{ 
\prod\displaylimits_{p \nmid d}
\Big(1 + \frac{r_D(p)\chi_D(p)}{\sqrt{p}} \Big)^2 
 \Big. \nonumber \\
&&
\Big.
\phantom{mmmmm} + \phantom{m}
O\Big( 
\big(\frac{d}{N}\big)^{\alpha}
\prod_{p \nmid d}\Big(1 ~+~ \frac{r_D(p)\chi_D(p)(p^{\alpha} + 1) }{\sqrt{p}}
+~
\frac{r_D(p)^2\chi_D(p)^2p^{\alpha}}{p} 
\Big)
\Big)
\Big\} \nonumber \\
&=&
\prod\displaylimits_{p}
\Big(1+ \frac{2r_D(p)\chi_D(p)}{\sqrt{p}} + r_D(p)^2
\Big(1+\frac{1}{p}\Big)
\Big) \nonumber\\
&&
\phantom{mmmmmm} + \phantom{m}
O\Big(N^{-\alpha}
\prod\displaylimits_{p}
\Big(1 ~+~ \frac{r_D(p)\chi_D(p) (p^{\alpha}+1)}{\sqrt{p}}
~+~
r_D(p)^2p^{\alpha}
\Big(1+\frac{1}{p}\Big)
\Big)
\Big).
\end{eqnarray}
The ratio of the error term to the main term in the above expression is
\begin{equation}\label{noname}
N^{-\alpha}\prod\displaylimits_{p}
\Bigg(1+\frac{\frac{r_D(p)\chi_D(p)}{\sqrt{p}}(p^{\alpha}-1)
+ r_D(p)^2\big(1+\frac{1}{p}\big)(p^{\alpha}-1)}
{1+\frac{2r_D(p)\chi_D(p)}{\sqrt{p}} + r_D(p)^2\big(1+\frac{1}{p}\big)}\Bigg).
\end{equation}
Choosing $\alpha=(\log L)^{-3}$ and proceeding as before, we see that
\eqref{noname} is
\begin{equation}\label{c3}
~\ll~
\exp\Big(-\frac{\alpha \log N}{\log\log N}\Big)
\end{equation}
for sufficiently large $N$.
Combining \eqref{c1}, \eqref{final}, \eqref{c2} and \eqref{c3}, we
get the desired asymptotic \eqref{big}.

Finally we consider the ratio 
$$
\prod_p \Bigg(1 +  \Big( \frac{r_D(p)^2}{p} + \frac{2r_D(p)\chi_D(p)}{\sqrt{p}} \Big) 
(1 + r_D(p)^2 )^{-1}   \Bigg)
$$
of \eqref{big} to \eqref{small} as $k \to \infty$.
Using the fact that $\log(1+x) = x-\frac{x^2}{2} + O(x^3)$ when $x <1$ and 
the binomial expansion of $(1+x)^{-1}$, we get
$$
\log \Bigg(1 +  \Big( \frac{r_D(p)^2}{p} + \frac{2r_D(p)\chi_D(p)}{\sqrt{p}} \Big) 
(1 + r_D(p)^2 )^{-1}   \Bigg) 
~=~ 
\frac{2 r_D(p)\chi_D(p)}{\sqrt{p}} ~-~ \frac{ r_D(p)^2}{p} 
~+~O\Big(\frac{ r_D(p)^3\chi_D(p)}{\sqrt{p}}\Big).
$$
Taking sum over all primes $p$ with $\pL$, we get
$$
\sum_{\pL} \Big( \frac{2 L}{p\log p} 
~-~ \frac{L^2}{p^2\log^2 p} 
~+~ O\Big(\frac{ L^3}{p^2\log^3 p}\Big) \Big)
~=~
(1+o(1) ) \sqrt{\frac{4\log N}{\log\log N}}~.
$$
Choosing $N = \sqrt{k|D|}/200$, the ratio becomes 
$$
\exp\Big(  ( 1 + o(1) ) \sqrt{ \frac{2\log k|D|}{\log\log k|D|} } \Big).
$$
\end{proof}
The following proposition plays an important role in the proof of our theorem.

\begin{prop}\label{4thmomentq}
For $k$ sufficiently large, we have
$$ 
\frac{12}{k-1}\sum_{f \in B_k(|D|)}\frac{|R(f \otimes \chi_D)|^4}{\omega(f)^*} 
~\ll~ 
\prod_{p}(1+r_D(p)^2)(1+2r_D(p)^2).  
$$
\end{prop}

\begin{proof}
When $N \leq \sqrt{k|D|}/(40\pi)$, we have
$$
\frac{12}{k-1}\sum_{f \in B_k(|D|)}
\frac{|R(f \otimes \chi_D)|^4}{\omega(f)^*} 
~=~
\frac{12}{k-1}\sum_{f \in B_k(|D|)}\frac{1}{\omega(f)^*}
\sum_{\substack{m_1,m_2\leq N\\ {m_3,m_4 \leq N}}}
\prod_{i=1}^4r_D(m_i)\lambda_f(m_i).
$$
Using Hecke relations and applying \eqref{PET}, we get
\begin{eqnarray}\label{HR}
&=&
\frac{12}{k-1}\sum_{f \in B_k(|D|)}\frac{1}{\omega(f)^*}
\sum_{\substack{m_1,m_2\leq N\\ {m_3,m_4 \leq N}}}
~\prod_{i=1}^4r_D(m_i)
\sum_{d_1\vert(m_1,m_2) \atop (d_1, D)=1}\lambda_f\left(\frac{m_1m_2}
{d_1^2}\right)\sum_{d_2\vert(m_3,m_4) \atop (d_2, D) =1}\lambda_f\left(\frac{m_3m_4}
{d_2^2}\right) \nonumber\\
&=&
\sum_{\substack{m_1,m_2\leq N\\ {m_3,m_4 \leq N}}}
~\prod_{i=1}^4r_D(m_i)
\sum_{d_1\vert(m_1,m_2)\atop{ {d_2\vert(m_3,m_4) \atop (d_1d_2, D)=1}}}
\Big\{\delta_{\frac{m_1m_2}{d_1^2},\frac{m_3m_4}{d_2^2}}+O(e^{-k})\Big\}.
\end{eqnarray}
By Cauchy-Schwarz inequality, the error term 
in \eqref{HR} becomes
$$
O\Big( e^{-k} N^4 \prod_{p}(1 + r_D(p)^2)^2 \Big)
=
O\Big( e^{-k} (kD)^4 \prod_{p}(1 + r_D(p)^2)^2 \Big).
$$
Let $(m_1m_2, m_3m_4)=st^2$ with $(s,t)=1$. Since $m_i$ is square free 
for each $1 \leq i \leq 4$, we can write 
$m_1m_2=ust^2$ and $m_3m_4=vst^2$ with $(u,v)=1$, $(u,t)=1$ 
and $(v,t)=1$. If there exists $d_1\vert (m_1,m_2)$ and $d_2\vert (m_3,m_4)$
such that $\frac{m_1m_2}{d_1^2}=\frac{m_3m_4}{d_2^2}$ then 
$p\vert d_1$ if and only if $p\vert tu$ and $p\vert d_2$ if and only if 
$p \vert tv$. This is because $(u,v)=1$ and hence $(s,d_1)=1=(s,d_2)$.
Also, since $(u,t)=1=(v,t)$, both $u$ and $v$ are perfect squares, 
say $u=u_1^2$ and $v=v_1^2$. Also, the equality 
$\frac{m_1m_2}{d_1^2}=\frac{m_3m_4}{d_2^2}$ imposes the conditions
$u_1\vert d_1$ and $u_2\vert d_2$.
Therefore the main term in \eqref{HR} becomes
$$
\sum_{t\leq N}r_D(t)^4\sigma_0(t)
\sum_{\substack{s \leq N^2/t^2 \\(t,s)=1}} r_D(s)^2
\sum_{\substack{u_1\leq \frac{N}{t\sqrt{s}}\\(u_1,st)=1}}r_D(u_1)^2
\sum_{\substack{v_1\leq \frac{N}{t\sqrt{s}}\\(v_1, u_1st)=1}}r_D(v_1)^2 
~\ll~
\prod_{p}(1+3r_D(p)^2+2r_D(p)^4), 
$$
where $\sigma_0(n)$ is the number of distinct divisors of $n$.
This completes the proof of the proposition.
\end{proof}

We next calculate the upper and lower bounds of 
$\displaystyle\frac{12}{k-1}\sum_{f \in S(k,|D|)}\frac{R(f \otimes \chi_D)^2}{\omega(f)^*}
L\big(1/2, f \otimes \chi_D \big)$.
We start with the following lower bound.
\begin{lem}\label{lower2}
Let $k$ be an even integer and $D$ an odd fundamental discriminant
with $\chi_D(-1)=i^k$. As $k$ goes to infinity, we have
\begin{equation}
\frac{12}{k-1}\sum_{f \in S(k,|D|)}\frac{R(f \otimes \chi_D)^2}{\omega(f)^*}
L\Big(1/2, f \otimes \chi_D \Big)
~\gg~ 
\prod_p \Big(1+r_D(p)^2\Big)
\exp\Bigg(10^{-3}\sqrt{\frac{\log (k |D|) }{\log\log (k|D|)}}\Bigg),
\end{equation}
where the implied constant in $\gg$ is absolute.
\end{lem}

\begin{proof}
Applying \lemref{chiasym} for $N = \sqrt{k|D|}/200$, we have
\begin{align*}
\frac{12}{k-1}\sum_{f \in S(k,|D|)}\frac{R(f \otimes \chi_D)^2}{\omega(f)^*}
L\left(\frac12, f \otimes \chi_D \right)
&\gg~ 
\prod_p \Big(1+r_D(p)^2 \Big(1+\frac{1}{p}\Big) 
+ \frac{2r_D(p) \chi_D(p)}{\sqrt{p}}\Big)\\
&\hspace{2mm}-~~
\prod_p \big(1 + r_D(p)^2 \big)
\exp\Bigg(1.41\sqrt{\frac{\log(k|D|)}{\log\log(k|D|)}}\Bigg)\\
&\gg \prod_p \Big(1 + r_D(p)^2 \Big)
\exp\Bigg(10^{-3}\sqrt{\frac{\log(k|D|)}{\log\log(k|D|)}}\Bigg).
\end{align*}
\end{proof}

We divide the proof of the upper bound in two parts.
For $D=1$, we use the result of Frolenkov~\cite{DF} while for 
$D \ne 1$, we use the result of Young \cite{MY}.

\begin{lem}\label{upper}
For $k$ sufficiently large, we have
\begin{equation*}
\frac{12}{k-1}\sum_{f \in \mathcal{S}(k)}\frac{R(f)^2}{\omega(f)^*}
 L\big(1/2, f \big)
~\ll~
\Big(\frac{ |\mathcal{S}(k)| \log^{11}k}{k} \Big)^{1/6} 
\prod_{p}(1 + r(p)^2)^{1/2}(1 +  2r(p)^2 )^{1/2}.
\end{equation*}
\end{lem}

\begin{proof}
Using Cauchy-Schwarz and H\"older's inequality, we have 
\begin{eqnarray*}
\frac{12}{k-1}\sum_{f \in \mathcal{S}(k)}\frac{R(f)^2}{\omega(f)^*}
 L\big(1/2,  f \big)
&\leq& 
\frac{12}{k-1}\Big(\sum_{f \in H_k}\frac{ |R(f)|^4}{\omega(f)^{*}}
\Big)^{1/2}\Big(\sum_{f \in H_k} \frac{1_{\mathcal{S}(k)}(f) 
L\big(1/2, f \big)^2}{\omega(f)^*} \Big)^{1/2}\\
&\leq&
\frac{12}{k-1}
\Big(
\sum_{f \in H_k}\frac{ |R(f)|^4}{\omega(f)^*}
\Big)^{1/2} \Big(\frac{|\mathcal{S}(k)|}{\omega(f)^*} \Big)^{1/6}
\Big( 
\sum_{f \in H_k} \frac{ L\big(1/2, f \big)^{3}}{\omega(f)^*}
\Big)^{1/3}.
\end{eqnarray*}
Applying \thmref{fro}, equation \eqref{omega2} and \propref{4thmomentq}, 
as $k \to \infty$, we get
$$
\frac{12}{k-1}\sum_{f \in \mathcal{S}(k)}\frac{R(f)^2}{\omega(f)^*}
 L\big(1/2,  f \big)
~\ll~
\Big(\prod_{p}(1+r(p)^2)(1+2r(p)^2)\Big)^{1/2}
|\mathcal{S}(k)|^{1/6}k^{-1/6} \log^{11/6}k,
$$
which gives the desired bound.
\end{proof}

\begin{lem}\label{upper2}
For any $\epsilon >0$, sufficiently large weight $k$ and an odd fundamental 
discriminant $D$ with $\chi_D(-1)=i^k$, we have
$$
\frac{12}{k-1}\sum_{f \in \mathcal{S}(k,D)}
\frac{R(f \otimes \chi_D)^2}{\omega(f)^*}
 L\big(1/2, f\otimes\chi_D \big)
~\ll_{\epsilon, D}~
\Big(\frac{ |\mathcal{S}(k,D)| }{k^{1 - \epsilon}} \Big)^{\frac16}
\Big(
\prod_{p \nmid D}(1 + r(p)^2)
(1+ 2r(p)^2)
\Big)^{\frac12}.
$$
\end{lem}

\begin{proof}
Using Cauchy-Schwarz inequality and H\"older's inequality, we have 
\begin{eqnarray*}
&&
\frac{12}{k-1}\sum_{f \in \mathcal{S}(k,D)}
\frac{R(f \otimes \chi_D)^2}{\omega(f)^*}
 L\left(1/2,f\otimes\chi_D\right) \\
&\leq& 
\frac{12}{k-1}
\Big(
\sum_{f \in B_k(|D|)}
\frac{|R(f \otimes \chi_D)|^4}{\omega(f)^{*2}}
\Big)^{1/2}
\Big(\sum_{f \in B_k(|D|)}1_{\mathcal{S}(k,D)}(f) 
L\big(1/2, f \otimes \chi_D \big)^2 \Big)^{1/2}\\
& \leq &
\frac{12}{k-1}
\Big(\sum_{f \in B_k(|D|)}
\frac{|R(f \otimes\chi_D)|^4}{\omega(f)^{*2}}
\Big)^{1/2}
|\mathcal{S}(k,D)|^{1/6}
\Big( \sum_{f \in B_k(|D|)}
L\big(1/2, f\otimes\chi_D \big)^{3}\Big)^{1/3}.
\end{eqnarray*}
Using \thmref{young}, equation \eqref{omega} and \propref{4thmomentq}, 
for any $\epsilon >0$, as $k \to \infty$, we get
$$
\frac{12}{k-1}\sum_{f \in \mathcal{S}(k,D)}
\frac{R(f \otimes \chi_D)^2}{\omega(f)^*}
 L\big(1/2, f \otimes \chi_D \big) 
~\ll_{\epsilon, D}~
\Big(\frac{ |\mathcal{S}(k,D)|}{k^{1-\epsilon}} \Big)^{1/6}
\Big(
\prod_{p \nmid D}(1 + r(p)^2)(1 + 2r(p)^2) \Big)^{1/2}.
$$
This completes the proof of \lemref{upper2}.
\end{proof}

\smallskip

\section{Proofs of \thmref{thm1}, \thmref{thm2} and \corref{cor1}}

\medskip

\begin{proof}[Proof of \thmref{thm1}]
Using \lemref{lower2} and \lemref{upper} and 
choosing $N = \frac{\sqrt{k}}{200}$, we have
\begin{equation*}
\prod_p\left(1+r(p)^2\right)\exp\Big(10^{-3}\sqrt{ \frac{\log k}{\log\log k} } \Big)
~\ll~
\Big( \frac{|\mathcal{S}(k)|  \log^{11}k}{k} \Big)^{1/6}
\Big(\prod_{p}(1+r(p)^2)(1+2r(p)^2)\Big)^{1/2}
\end{equation*}
as $k \to \infty$. This implies that
$$
\frac{ k^{1/6} }{\log^{11/6}k}
\exp\Big(10^{-3}\sqrt{ \frac{\log k}{\log\log k} } \Big)
~\ll~ 
|\mathcal{S}(k)|^{1/6}
\prod_{p}\Big(1+\frac{r(p)^2}{1+r(p)^2}\Big)^{1/2}
~\ll~ 
|\mathcal{S}(k)|^{1/6}
\exp\Big(\frac{\log{k}}{4\log\log{k}}\Big)
$$
as $k \to \infty$.
Therefore we get
$$
|\mathcal{S}(k)| ~\gg ~ \frac{k}{ (\log k)^{11}  \exp\Big(\frac{3\log k}{2\log\log k} \Big)}
$$
as $k \to \infty$.
\end{proof}

\smallskip

\begin{proof}[Proof of \thmref{thm2}]
Using \lemref{lower2} and \lemref{upper2}, for any $\epsilon >0$ and 
$N = \frac{\sqrt{k |D|}}{200}$, we get 
\begin{equation*}
\prod_{p \nmid D} \left(1 + r(p)^2\right)
\exp\Big(10^{-3}\sqrt{ \frac{\log k}{\log\log k} } \Big)
~\ll_{\epsilon, D}~
\Big(\frac{|\mathcal{S}(k, D)| }{k^{1 - \frac{6\epsilon}{7}}} \Big)^{1/6}
\Big(\prod_{p \nmid D}(1 + r(p)^2) (1+ 2r(p)^2)\Big)^{1/2}
\end{equation*}
as $k \to \infty$. This implies that
\begin{eqnarray*}
k^{\frac16 - \frac{\epsilon}{7}}
\exp\Big(10^{-3}\sqrt{ \frac{\log k}{\log\log k} } \Big)
&\ll_{\epsilon, D}&
|\mathcal{S}(k,D)|^{1/6}
\prod_{p \nmid D}\Big(1+\frac{r(p)^2}{1 + r(p)^2}\Big)^{1/2} \\
&\ll_{\epsilon, D}&
|\mathcal{S}(k,D)|^{1/6}
~\exp\Big(\frac{\log k}{4\log\log k}\Big)
\end{eqnarray*}
as $k \to \infty$. Hence we get
$$
|\mathcal{S}(k,D)|  ~\gg_{\epsilon, D}~ k^{1-\epsilon}
$$
as $k \to \infty$.
\end{proof}

\smallskip
\begin{proof}[Proof of \corref{cor1}]
By \rmkref{reqcor} and equation \eqref{2.51}, for any $\epsilon >0$,
there are $\gg_{\epsilon, D} k^{1 - \epsilon}$ Hecke eigenforms 
$g \in S_{(k+1)/2}^+(4)$ for which 
$$
c_g(|D|)^2
\geq
\frac{\Gamma(k/2)}{\pi^{k/2}}|D|^{(k-1)/2} 
\exp\left(1.41\sqrt{ \frac{ \log(k|D|) }{\log\log(k|D|) } }\right).
$$
This completes the proof of the corollary.
\end{proof}

\bigskip
\noindent
{Acknowledgement.} We would like to thank the referee for careful reading of 
the article which improved the exposition.


\smallskip

\begin{thebibliography}{100}
\bibitem{AL}
A. Atkin and J. Lehner,
{\em Hecke operators on $\Gamma _{0}(m)$},
Math. Ann. {\bf 185} (1970), 134--160.

\bibitem{CFKRS}
J. Conrey, D. Farmer, J. Keating,  M. Rubinstein and N. Snaith, 
{\em Integral moments of L-functions},
Proc. London Math. Soc. (3) {\bf 91} (2005), no.1, 33--104.

\bibitem{CI}
J. Conrey and H. Iwaniec, 
{\em The cubic moment of central values of automorphic L-functions},
Ann. of Math. (2) {\bf 151} (2000), no. 3, 1175--1216.

\bibitem{DF}
D. Frolenkov,
{\em The cubic moment of automorphic L-functions in the weight aspect},
J. Number Theory {\bf 207} (2020), 247--281.

\bibitem{GKS}
S. Gun, W. Kohnen and K. Soundararajan,
{\em Large Fourier coefficients of half-integer weight modular forms},
to appear in Amer. J. Math.
https://preprint.press.jhu.edu/ajm/sites/default/files/AJM-gun-kohnen-soundara.pdf

\bibitem{HL}
J. Hoffstein and P. Lockhart,
{\em Coefficients of Maass forms and the Siegel zero},
Ann. of Math. (2) {\bf 140} (1994), no.1, 161--181.

\bibitem{HI}
H. Iwaniec,
{\em Fourier coefficients of modular forms of half-integral weight},
Invent. Math. {\bf 87} (1987), no.2, 385--401.

\bibitem{ILS}
H. Iwaniec, W. Luo and P. Sarnak, 
{\em Low lying zeros of families of L-functions},
Inst. Hautes \'Etudes Sci. Publ. Math. {\bf 91} (2000),
55--131.

\bibitem{IS}
H. Iwaniec and P. Sarnak, 
{\em The non-vanishing of central values of automorphic 
$L$-functions and Landau-Siegel zeros},
Isr. J. Math. {\bf 120} (2000), 155--177.

 
\bibitem{JLS}
J. J{\" a}{\" a}saari, S. Lester, and A. Saha, 
{\em On fundamental Fourier coefficients of Siegel cusp forms of degree $2$}, 
J. Inst. Math. Jussieu {\bf 22} (2023), no.4, 1819--1869.

\bibitem{KeS}
 J. Keating and N. Snaith, 
{\em Random matrix theory and L-functions at $s=1/2$},
Comm. Math. Phys. {\bf 214} (2000), no.1, 91--110.

\bibitem{WK}
W. Kohnen, 
{\em Modular forms of half-integral weight on $\Gamma_0(4)$}, 
Math. Ann.  {\bf 248} (1980), no. 3, 249--266.

\bibitem{KZ}
W. Kohnen and D. Zagier, 
{\em Values of $L$-series of modular forms at the center of the critical strip}, 
Invent. Math. {\bf 64} (1981), 175--198.

\bibitem{WK1}
W. Kohnen, 
{\em Newforms of half-integral weight}, 
J. Reine Angew. Math. {\bf 333} (1982), 32--72. 

\bibitem{WK85}
W. Kohnen,
{\em Fourier coefficients of modular forms of half-integral weight},
Math. Ann. {\bf 271} (1985), no.2, 237--268.

\bibitem{LT}
Y.-K. Lau and K.-M.Tsang,
{\em A mean square formula for central values of twisted automorphic $L$-functions},
Acta Arithmetica  {\bf 118} (2005), no. 3, 231--262.

\bibitem{WL}
W. Luo,
{\em Nonvanishing of the central L-values with large weight},
Adv. Math. {\bf 285} (2015), 220--234.

\bibitem{MM}
M. R. Murty, 
Problems in analytic number theory, Graduate Texts in Mathematics {\bf 206},
Second Edition, {\em Springer}, (2008).

\bibitem{SN}
S. Niwa, 
{\em Modular forms of half integral weight and the integral of certain theta-functions},
Nagoya Math. J. {\bf 56} (1975), 147--161.

\bibitem{ZP}
Z. Peng,
{\em Zeros and central values of automorphic L-functions},
Ph.D. Thesis, Princeton University (2001).

\bibitem{MR}
M. Razar,
{\em Modular Forms for $G_0(N)$ and Dirichlet Series},
Trans. Amer. Math. Soc.{\bf 231} (1977), no.2, 489--495.

\bibitem{RS}
Z. Rudnick and K. Soundararajan, 
{\em Lower bounds for moments of L-functions: symplectic and orthogonal examples}, 
In: Multiple Dirichlet Series, Automorphic Forms, and Analytic Number Theory, 
Proc. Sympos. Pure Math. {\bf 75} (2006), 293--303.

\bibitem{GS}
G. Shimura, {\em On modular forms of half integral weight},
Ann. of Math. {\bf 97} (1973), no. 2, 440--481.

\bibitem{KS}
K. Soundararajan,
{\em Extreme values of zeta and L-functions},
Math. Ann. {\bf 342} (2008), no. 2, 467--486.

\bibitem{JW}
J. L. Waldspurger, 
{\em Sur les coefficients de Fourier des formes modulaires de poids demi-entier},  
J. Math. Pures Appl. {\bf 60} (1981), 375--484.

\bibitem{MY}
M. Young, 
{\em Weyl-type hybrid subconvexity bounds for twisted L-functions 
and Heegner points on shrinking sets},
J. Eur. Math. Soc. (JEMS) {\bf 19} (2017), no. 5, 1545--1576.

\end{thebibliography}
\end{document}